\documentclass{amsart}
\usepackage{graphicx}
\usepackage{amscd}
\usepackage{amsmath}
\usepackage{amsfonts}
\usepackage{amssymb}
\newtheorem{theorem}{Theorem}
\theoremstyle{plain}

\newtheorem{corollary}{Corollary}

\newtheorem{definition}{Definition}

\newtheorem{lemma}{Lemma}

\newtheorem{remark}{Remark}

\numberwithin{equation}{section}

\begin{document}
\title[Sobolev inequalities in metric spaces]{Symmetrization and sharp Sobolev inequalities in metric spaces }
\author{Jan Kali\v s}
\address{Department of Mathematics\\
Florida Atlantic University\\
Boca Raton, Fl 33431}
\email{kalis@math.fau.edu}
\author{Mario Milman}
\address{Department of Mathematics\\
Florida Atlantic University\\
Boca Raton, Fl 33431}
\email{extrapol@bellsouth.net}
\urladdr{http://www.math.fau.edu/milman}

\begin{abstract}
We derive sharp Sobolev inequalities for Sobolev spaces on metric spaces. In
particular, we obtain new sharp Sobolev embeddings and Faber-Krahn estimates
for H\"{o}rmander vector fields.
\end{abstract}\maketitle

\section{Introduction}

Recently, a rich theory of Sobolev spaces on metric spaces has been developed
(cf. \cite{HK}, \cite{HKF}, and the references therein). In particular, this
has led to the unification of some aspects of the classical theory of Sobolev
spaces with the theory of Sobolev spaces of vector fields satisfying
H\"{o}rmander's condition. At the root of these developments are suitable
forms of Poincar\'{e} inequalities which, in fact, can be used to provide a
natural method to define the notion of a gradient in the setting of metric
spaces. In the theory of H\"{o}rmander vector fields the relevant Poincar\'{e}
inequalities had been obtained much earlier by Jerison \cite{JE}:
\begin{equation}
\left(  \frac{1}{\left|  B\right|  }\int_{B}\left|  f-f_{B}\right|
^{2}dx\right)  ^{1/2}\leq Cr(B)\left(  \frac{1}{\left|  B\right|  }\int
_{B}\left|  Xf\right|  ^{2}dx\right)  ^{1/2}, \label{vale}%
\end{equation}
where $X=(X_{1},\dots,X_{m})$ is a family of $\mathcal{C}^{\infty}$
H\"{o}rmander vector fields, $\left|  Xf\right|  =\left(  \sum\left|
X_{i}f\right|  ^{2}\right)  ^{1/2},$ $dx$ is Lebesgue measure, $B$ is a ball
of radius $r(B)$, with respect to the Carnot-Carath\'{e}odory metric. For more
on the connection between the theory of Sobolev spaces on metric spaces and
Sobolev spaces on\ Carnot groups we refer to the Appendix below and \cite{HKF}.

The purpose of this paper is to prove sharp forms of the classical Sobolev
inequalities in the context of metric spaces. In fact, we develop an approach
to symmetrization in the metric setting which has applications to other
problems as well. In particular, we will show some functional forms of the
Faber-Krahn inequalities which are new even in the classical setting.

A well known, and very natural, approach to the Sobolev inequalities is
through the use of the isoperimetric inequality and related rearrangement
inequalities (for an account cf. \cite{Ta}). For example, a good deal of the
classical inequalities can be in fact derived from (cf. \cite{BMR}, \cite{Ko},
and also \cite{JMMM})
\begin{equation}
\frac{1}{t}\int_{0}^{t}[f^{\ast}(s)-f^{\ast}(t)]ds\leq ct^{1/n}\left(
\frac{1}{t}\int_{0}^{t}\left|  \nabla f\right|  ^{\ast}(s)ds\right)
,t>0,f\in\mathcal{C}_{0}^{\infty}(\mathbb{R}^{n}). \label{ser1}%
\end{equation}
For example, in \cite{BMR} and \cite{MP} it is shown how, starting from
(\ref{ser1}), one can derive Sobolev inequalities which are sharp, including
the borderline cases, within the class of Sobolev spaces based on
rearrangement invariant spaces. Therefore, it seemed natural to us to try to
extend (\ref{ser1}) to the metric setting. At the outset one obstacle is that
the usual methods to prove (\ref{ser1}) are not available for metric spaces
(cf. \cite{BMR}, \cite{JMMM}) . However, we noticed that, in the Euclidean
setting, (\ref{ser1}) is the rearranged version of a Poincar\'{e} inequality.
More specifically, suppose that $f$ and $g$ are functions such that, for any
cube $Q\subset\mathbb{R}^{n}$ with sides parallel to the coordinate axes, we
have
\begin{equation}
\frac{1}{\left|  Q\right|  }\int_{Q}\left|  f(x)-f_{Q}\right|  dx\leq
c\frac{\left|  Q\right|  ^{1/n}}{\left|  Q\right|  }\int_{Q}g(x)dx.
\label{vale1}%
\end{equation}
Then, the following version of (\ref{ser1}) holds,
\begin{equation}
\frac{1}{t}\int_{0}^{t}[f^{\ast}(s)-f^{\ast}(t)]ds\leq ct^{1/n}\left(
\frac{1}{t}\int_{0}^{t}g^{\ast}(s)ds\right)  . \label{ser6}%
\end{equation}
By Poincar\'{e}'s inequality, (\ref{vale1}) holds with $g=|\nabla f|$ and
therefore the implication $(\ref{vale1})\Rightarrow(\ref{ser6})$ provides us
with a proof of (\ref{ser1}). This is somewhat surprising since the usual
proofs of (\ref{ser1}) depend on a suitable representation of $f$ in terms of
$\nabla f$. Since, in the context of metric spaces, the gradient is defined
through the validity of (\ref{vale1}), this is a crucial point for our
development of the symmetrization method in this setting.

Since the mechanism involved in transforming (\ref{vale1}) into (\ref{ser6})
plays an important role in our approach, it is instructive to present it here
in the somewhat simpler, but central, Euclidean case. The first step is to
reformulate (\ref{vale1}) as an inequality between maximal operators
\begin{align}
f_{1/n}^{\#}(x)  &  :=\sup_{Q\backepsilon x}\frac{1}{\left|  Q\right|
^{1+1/n}}\int_{Q}\left|  f(x)-f_{Q}\right|  dx\nonumber\\
&  \leq c\sup_{Q\backepsilon x}\frac{1}{\left|  Q\right|  }\int_{Q}%
g(x)dx=cMg(x), \label{ser3}%
\end{align}
where $M$ is the non-centered maximal operator of Hardy-Littlewood. At this
point taking rearrangements on both sides of (\ref{ser3}) leads to
\begin{equation}
(f_{1/n}^{\#})^{\ast}(t)\leq cMg^{\ast}(t)\leq cg^{\ast\ast}(t). \label{ser4}%
\end{equation}
Here the estimate for the maximal operator of Hardy-Littlewood is a well
known, and easy, consequence of the fact that $M$ is weak type $(1,1)$ and
strong type $(\infty,\infty)$. Moreover, by a simple variant of an inequality
of Bennett-DeVore-Sharpley \cite{BDS}, we have%
\begin{equation}
\left(  f^{\ast\ast}(t)-f^{\ast}(t)\right)  t^{-1/n}\leq c(f_{1/n}^{\#}%
)^{\ast}(t). \label{ser5}%
\end{equation}
Combining (\ref{ser4}) and (\ref{ser5}), we see that if (\ref{vale1}) holds
then (\ref{ser6}) holds.

The method of proof outlined above can be developed in more general settings
as long as suitable variants of the classical covering lemmas, which are
needed to estimate the underlying maximal operators, are available. In the
context of metric spaces the covering lemmas we need\footnote{In the PhD
thesis of the first author \cite{Ka}, the required covering lemmas been
obtained for domains with Lipschitz boundaries.} were obtained in \cite{LP}.
Once the rearrangement inequalities are at hand we can use standard machinery
to derive suitable Sobolev inequalities (see Section 3).

To give a more precise description of the contents of this paper we now recall
the definition of a $p-q-$Poincar\'{e} inequality. In what follows $(X,\mu)$
is a homogenous metric space\footnote{See Definition
\ref{de:homogeneous_space} below.} with a doubling Borel measure $\mu$ of
dimension $s$.

\begin{definition}
\label{de:metric_poincare}(cf. \cite{HK}, \cite{HKF}) Let $\Omega$ be a
measurable subset of $X,$ and let $f$ and $g$ be measurable functions defined
on $\Omega$, with $g\geq0.$ Let $p,q\geq1.$ We shall say that $f$ and $g$
satisfy a $p-q-$Poincar\'{e} inequality$,$ if for some constants $c_{P}>0$,
$\sigma\geq1,$%
\begin{equation}
\left(  \frac{1}{\mu(B)}\int_{B}|f(x)-f_{B}|^{p}d\mu(x)\right)  ^{1/p}\leq
c_{P}r(B)\left(  \frac{1}{\mu(\sigma B)}\int_{\sigma B}g^{q}(x)d\mu(x)\right)
^{1/q} \label{pqP}%
\end{equation}
holds for every ball $B$ such that $\sigma B\subset\Omega$, where
$f_{B}=\left(  \mu(B)\right)  ^{-1}\int_{B}f(x)d\mu(x)$. We may then refer to
$f$ as a ($p-q-$) Sobolev function and to $g$ as its gradient.
\end{definition}

We can now state our main results. We start with the following extension of
(\ref{ser1}).

\begin{theorem}
(cf. Theorem \ref{th:metric_bi} below) Let $B_{0}\subset X$ be a ball, and
suppose that $f$ and $g$ satisfy a $p-q-$Poincar\'{e} inequality on $4\sigma
B_{0}$. Then there exist constants $c_{1}>0,0<c_{2}\leq1,$ independent of
$B_{0}$, $f$ and $g,$ such that
\begin{equation}
t^{-\frac{1}{s}}\left(  \frac{1}{t}\int_{0}^{t}\left(  [f\chi_{B_{0}}]^{\ast
}(s)-[f\chi_{B_{0}}]^{\ast}(t)\right)  ^{p}ds\right)  ^{1/p}\leq c_{1}%
[(g^{q})^{\ast\ast}(t)]^{1/q},0<t<c_{2}\mu(B_{0}). \label{pq-bi}%
\end{equation}
\end{theorem}

Following \cite{MP}, given a r.i. space $Y,$ we introduce the spaces
$Y^{p}(\infty,s)$ (see Section 2 below) that contain all the functions for
which the $Y$-norm of the expression on the left-hand side of (\ref{pq-bi}) is
finite. The following sharp Sobolev embedding theorem then follows immediately.

\begin{theorem}
(cf. Theorem \ref{th:metric_embedding} below) Let $B_{0}\subset X$ be a ball,
and let $Y(X)$ be an r.i. space. Suppose that the operator $P_{\max\{p,q\}}$
(cf. (\ref{ser8}) below) is bounded on $Y(X)$. Then, if $f$ and $g$ satisfy a
$p-q-$Poincar\'{e} inequality on $4\sigma B_{0}$ with constant $c_{P}$, there
exists a constant $c=c(B_{0},c_{P},p,q,Y)>0$ such that
\[
\Vert f\chi_{B_{0}}\Vert_{Y^{p}(\infty,s)}\leq c(\Vert g\Vert_{Y}+\Vert
f\Vert_{Y}).
\]
\end{theorem}

We also provide a new application of our rearrangement inequality
(\ref{pq-bi}) to the study of the so called functional forms of the
Faber-Krahn inequalities in metric spaces (cf. Section 4 below). We now
illustrate these ideas in the classical Euclidean case. For example, using
$\int_{f^{\ast}(t)}^{\infty}\lambda_{f}(u)=t(f^{\ast\ast}(t)-f^{\ast}(t)),$
and H\"{o}lder's inequality, we see that (\ref{ser1}) implies\footnote{Here
and in what follows the symbol $\approx$ denotes equivalence modulo constants,
and the symbol $\preceq$ denotes smaller or equal modulo constants.}%
\begin{align*}
\int_{f^{\ast}(t)}^{\infty}\lambda_{f}(u) &  \preceq t^{1/n}\int_{0}%
^{t}\left|  \nabla f\right|  ^{\ast}(u)du\\
&  \preceq t^{1/n}\left(  \int_{0}^{t}\left|  \nabla f\right|  ^{\ast}%
(u)^{p}du\right)  ^{1/p}t^{1/p^{\prime}}.
\end{align*}
Now let $t=\Vert f\Vert_{0}=\int_{\{\left|  f\right|  >0\}}dx$ , and observe
that then $f^{\ast}(t)=0$, and $\int_{f^{\ast}(t)}^{\infty}\lambda
_{f}(u)=\left\|  f\right\|  _{1}.$ We have thus obtained the following
Faber-Krahn inequality%
\begin{equation}
\left\|  f\right\|  _{1}\preceq\Vert f\Vert_{0}^{1/n+1-1/p}\left\|  \nabla
f\right\|  _{p}.\label{samba}%
\end{equation}
More generally, if a $p-q-$Poincar\'{e} inequality holds then we can use
(\ref{pq-bi}) and a similar argument to prove\footnote{For a different
approach to (\ref{samba}) we refer to \cite{MV}. For a far reaching
generalization of Theorem \ref{marcao}, using the beautiful ideas of Jawerth
\cite{BJ}, see \cite{BJMM}.}

\begin{theorem}
\label{marcao}(cf. Theorem \ref{th:metric_faber} below) Let $B_{0}\subset X$
be a ball, and let $f$ be a function with $\text{supp}(f)\subset B_{0}$. Let
$Z(X)$ be an r.i. space and let $\phi_{Z^{\prime}}$ denote the fundamental
function of its associate space $Z^{\prime}(X).$

(i) Let $f$ be a $p-q-$Sobolev function and let $g$ be a gradient of $f$. If
$\Vert f\Vert_{0}<c_{2}\mu(B_{0})$, where $c_{2}$ is the constant of Theorem
\ref{th:metric_bi}, then
\begin{equation}
\Vert f\Vert_{L^{p}(X)}\leq c\left[  \Vert g^{q}\Vert_{Z(X)}\phi_{Z^{\prime}%
}(\Vert f\Vert_{0})\right]  ^{1/q}\Vert f\Vert_{0}^{\frac{s+p}{sp}-\frac{1}%
{q}}.
\end{equation}

(ii) Let $f$ be a $1-q-$Sobolev function, $q>s,$ and let $g$ be a gradient of
$f$. If $\Vert f\Vert_{0}<c_{2}\mu(B_{0})$, where $c_{2}$ is the constant of
Theorem \ref{th:metric_bi}, then
\begin{equation}
\Vert f\Vert_{L^{\infty}(X)}\leq c\left[  \Vert g^{q}\Vert_{Z(X)}%
\phi_{Z^{\prime}}(\Vert f\Vert_{0})\right]  ^{1/q}\Vert f\Vert_{0}^{1/s-1/q}.
\end{equation}
\end{theorem}

\section{The basic symmetrization inequality in metric spaces}

We start with a definition.

\begin{definition}
\label{de:homogeneous_space} An homogeneous space consists of a metric space
$X$ and a Borel measure $\mu$ on $X,$ such that $0<\mu(B(x,r))<\infty,$ for
all $x\in X,r>0,$ and, moreover, the measure $\mu$ satisfies a doubling
condition:
\begin{equation}
\mu(B(x,2r))\leq c_{d}\mu(B(x,r)), \label{eq:doubling}%
\end{equation}
for all $x\in X$ and $r>0$. If $c_{d}$ is the smallest constant in
(\ref{eq:doubling}) then the number $s=\log_{2}c_{d}$ is called the doubling
order or the dimension of $\mu$.
\end{definition}

\begin{remark}
Note that if we fix a ball $\widetilde{B}\subset X$, then by iterating
(\ref{eq:doubling}) (cf. Lemma 14.6 in \cite{HK}) we can find a positive
constant $c$ (possibly depending on $\widetilde{B}$) such that for every ball
$B\subset\widetilde{B}$ we have
\begin{equation}
\frac{\mu(B)}{\mu(\widetilde{B})}\geq cr(B)^{s}. \label{eq:meas_growth}%
\end{equation}
\end{remark}

In what follows given a ball $B=B(x,r)$, $\varrho B$ will denote the ball
concentric with $B,$ whose radius is $\varrho r$.

A rearrangement invariant (r.i.) space $Y=Y(X)$ is a Banach function space of
$\mu-$measurable functions on $X$ endowed with a norm $\Vert\cdot\Vert_{Y}$
such that if $f\in Y$ and $g^{\ast}=f^{\ast}$ then $g\in Y$ and $\Vert
g\Vert_{Y}=\Vert f\Vert_{Y}$. The fundamental function $\phi_{Y}$ of $Y$ is
defined for $t$ in the range of $\mu$ by
\[
\phi_{Y}(t)=\Vert\chi_{E}\Vert_{Y},
\]
where $E$ is any subset of $X$ with $\mu(E)=t$. Recall that any
resonant\footnote{See \cite[Definition II.2.3]{BS}.} r.i. space $Y$ has a
representation as a function space $Y{\symbol{94}}(0,\infty)$ such that (cf.
\cite[Theorem II.4.10]{BS})
\[
\left\|  f\right\|  _{Y(X)}=\left\|  f^{\ast}\right\|  _{Y{\symbol{94}%
}(0,\infty)}.
\]
Since the measure space will be always clear from the context it is convenient
to ``drop the hat'' and use the same letter $Y$ to indicate the different
versions of the space $Y$ that we use.

Let $P$ denote the usual Hardy operator $P:f(t)\mapsto t^{-1}\int_{0}%
^{t}f^{\ast}(s)ds.$ The operators $P_{p}$, $p\geq1,$ are defined by
\begin{equation}
(P_{p}f)(t)=[P((f^{\ast})^{p})(t)]^{1/p}. \label{ser8}%
\end{equation}

\begin{definition}
Let $Y$ be a r.i. space, and let $p\geq1$ and $r>0$. Let
\[
Y^{p}(\infty,r)=\{f:\Vert f\Vert_{Y^{p}(\infty,r)}=\Vert t^{-1/r}\left(
\frac{1}{t}\int_{0}^{t}[f^{\ast}(s)-f^{\ast}(t)]^{p}ds\right)  ^{1/p}\Vert
_{Y}<\infty\}.
\]
\end{definition}

\begin{remark}
Under suitable assumptions the expression defining the ``norm''\footnote{In
\cite{CGMP} sets similar to $Y^{p}(\infty,r)$ are considered and conditions
are given for these sets to be equivalent to normed spaces.} of the
$Y^{p}(\infty,r)$ spaces can be simplified. For example, suppose that $p$ and
$r$ are such that
\begin{equation}
1\leq p<\frac{p_{Y}r}{p_{Y}-r}, \label{eq:space_coin_ass2}%
\end{equation}
where $p_{Y}$ is the lower Boyd index $p_{Y}$ of $Y,$ and suppose, moreover
(cf. \cite{MP}),
\begin{equation}
\int_{1}^{\infty}s^{1/r}d_{Y}\left(  \frac{1}{s}\right)  \frac{ds}{s}<\infty,
\label{eq:space_coin_ass1}%
\end{equation}
where $d_{Y}(s)$ is the norm of the dilation operator $D_{s}:f(\cdot)\mapsto
f(\cdot s).$ Then, for $f$ with $f^{\ast\ast}(\infty)=0$, we have
\[
\Vert f\Vert_{Y^{p}(\infty,r)}\approx\Vert f\Vert_{Y^{1}(\infty,r)}.
\]
\end{remark}

\begin{proof}
From (\ref{eq:space_coin_ass1}) it follows that, for $f^{\ast\ast}(\infty)=0,$
we have (cf. \cite[Lemma 2.6]{MP}),
\begin{equation}
\Vert t^{-1/r}(f^{\ast\ast}(t)-f^{\ast}(t))\Vert_{Y}\approx\Vert
t^{-1/r}f^{\ast\ast}(t)\Vert_{Y}. \label{eq:space_coin_ineq1}%
\end{equation}
On the other hand, since $\Vert D_{a}[f(t)t^{-1/r}]\Vert_{Y}=a^{-1/r}%
\Vert\lbrack D_{a}f(t)]t^{-1/r}\Vert_{Y}$, we have
\[
\Vert D_{a}[f(t)t^{-1/r}]\Vert_{Y}\leq ca^{-1/p}\Vert f(t)t^{-1/r}\Vert_{Y}%
\]
if and only if
\[
\Vert D_{a}[f(t)]t^{-1/r}\Vert_{Y}\leq ca^{-1/p+1/r}\Vert f(t)t^{-1/r}%
\Vert_{Y}.
\]
Consequently, if we let $Y(t^{-1/r})$ be the space defined by the norm $\Vert
f(t)t^{-1/r}\Vert_{Y}$ then the lower Boyd index of $Y(t^{-1/r})$ is equal to
$p_{Y}r/(p_{Y}-r)$, where $p_{Y}$ is the lower Boyd index of $Y$. Now, in view
of (\ref{eq:space_coin_ass2}), with $r=p,$ it follows from \cite[Theorem 2
(i)]{MO}$,$ that the operator $P_{p}$ is continuous on $Y(t^{-1/r})$. Thus,
\[
\Vert\left(  \frac{1}{t}\int_{0}^{t}[f^{\ast}(s)]^{p}ds\right)  ^{1/p}%
t^{-1/r}\Vert_{Y}\leq c\Vert f^{\ast}(t)t^{-1/r}\Vert_{Y}\leq\Vert f^{\ast
\ast}(t)t^{-1/r}\Vert_{Y}.
\]
Combining the last inequality with (\ref{eq:space_coin_ineq1}) we obtain,%
\[
\Vert f\Vert_{Y^{p}(\infty,r)}\lesssim\Vert f\Vert_{Y^{1}(\infty,r)}.
\]
The reverse inequality follows readily from H\"{o}lder's inequality.
\end{proof}

The expression on the left-hand side of (\ref{ser3}) is a modification of the
well-known sharp maximal operator of Fefferman-Stein (cf. \cite{BS}) which is
defined for $f\in L_{\text{loc}}^{1}(X)$ and $p,q\geq1,$ by
\[
f_{B_{0},p,q}^{\#}(x)=\sup_{x\in B\subset B_{0}\text{ a ball}}\left(  \frac
{1}{\mu(B)^{q}}\int_{B}|f(y)-f_{B}|^{p}d\mu(y)\right)  ^{1/p}.
\]

For the proof of our symmetrization inequality we need the following version
of a covering lemma from \cite{LP}.

\begin{lemma}
\label{le:covering} There exist positive constants $c,\lambda,$ with
$\lambda<1,$ such that for any ball $B,$ and any open set $E\subset B$ with
$\mu(E)\leq\lambda\mu(B)$, there exists a countable family of balls
$\{B_{i}\}_{i=1}^{\infty}$ such that

\begin{enumerate}
\item $B_{i}\subset4B$ for $i=1,\dots$

\item $E\subset\bigcup_{i=1}^{\infty} B_{i}$

\item $\sum_{i}\mu(B_{i})\leq c \mu(E)$

\item $\mu(B_{i}\cap E)\leq(1/2)\mu(B_{i}\cap B)$ for $i=1,\dots$
\end{enumerate}
\end{lemma}

\begin{proof}
Follows readily from the proof of \cite[Lemma 3.1]{LP}.
\end{proof}

\begin{theorem}
\label{th:metric_bi_lhs}There exist positive constants $c_{1},c_{2}$, such
that, for any ball $B_{0}\subset X$, $p,q\geq1,$ and for all $f\in
L_{\text{loc}}^{1},$ $0<t<c_{2}\mu(B_{0}),$ we have%
\begin{equation}
t^{-q/p}\left(  \int_{0}^{t}([f\chi_{B_{0}}]^{\ast}(s)-[f\chi_{B_{0}}]^{\ast
}(t))^{p}ds\right)  ^{1/p}\leq c_{1}(f_{4B_{0},p,q}^{\#})^{\ast}(t).
\label{eq:metric_bi_lhs}%
\end{equation}
\end{theorem}

\begin{proof}
Follows along the lines of the corresponding proof in \cite[Theorem V.7.3]{BS}.

It suffices to establish (\ref{eq:metric_bi_lhs}) for nonnegative functions.
Let $\lambda$ be as in Lemma \ref{le:covering} and fix $0<t<\frac{\lambda}%
{3}\mu(B_{0}).$ Let
\[
E=\{x\in B_{0}:f(x)>[f\chi_{B_{0}}]^{\ast}(t)\}
\]
and
\[
F=\{x\in B_{0}:f_{4B_{0},p,q}^{\#}(x)>[f_{4B_{0},p,q}^{\#}\chi_{B_{0}}]^{\ast
}(t)\}.
\]
There exists an open set $\Omega\supset E\cup F,$ with measure at most
$3t\leq\lambda\mu(B_{0})$. Consequently, we can apply Lemma \ref{le:covering}
to obtain a family of balls $\{B_{j}\}_{j},$ such that all the conditions this
Lemma are verified. Define disjoint sets by letting $M_{1}=B_{1}$ and
$M_{k}=B_{k}\setminus\bigcup_{i=1}^{k-1}B_{i}$ for $k=2,\dots$. We have
\begin{align*}
\int_{0}^{t}([f\chi_{B_{0}}]^{\ast}(s)-[f\chi_{B_{0}}]^{\ast}(t))^{p}ds  &
=\int_{E}\{f(x)-[f\chi_{B_{0}}]^{\ast}(t)\}^{p}d\mu(x)\\
&  =\sum_{j=1}^{\infty}\int_{E\cap M_{j}}\{f(x)-[f\chi_{B_{0}}]^{\ast
}(t)\}^{p}d\mu(x)\\
&  \leq c\sum_{j=1}^{\infty}\int_{B_{j}}|f-f_{B_{j}}|^{p}d\mu(x)+\\
&  c\sum_{j=1}^{\infty}\mu(E\cap M_{j})\{f_{B_{j}}-[f\chi_{B_{0}}]^{\ast
}(t)\}_{+}^{p}\\
&  =c(\alpha+\beta)\text{, say.}%
\end{align*}
Now,%
\begin{align*}
\beta &  \leq\sum_{\{j:f_{B_{j}}>[f\chi_{B_{0}}]^{\ast}(t)\}}\mu(E\cap
B_{j})\{f_{B_{j}}-[f\chi_{B_{0}}]^{\ast}(t)\}^{p}\\
&  \leq\sum_{\{j:f_{B_{j}}>[f\chi_{B_{0}}]^{\ast}(t)\}}\mu(B_{0}\cap
B_{j}\setminus E)\{f_{B_{j}}-[f\chi_{B_{0}}]^{\ast}(t)\}^{p}\\
&  \leq\sum_{\{j:f_{B_{j}}>[f\chi_{B_{0}}]^{\ast}(t)\}}\int_{B_{0}\cap
B_{j}\setminus E}\{f_{B_{j}}-f(x)\}^{p}d\mu(x)\\
&  \leq\sum_{j=1}^{\infty}\int_{B_{j}}|f-f_{B_{j}}|^{p}d\mu(x)=\alpha.
\end{align*}
Combining the previous estimates we obtain
\[
\int_{0}^{t}([f\chi_{B_{0}}]^{\ast}(s)-[f\chi_{B_{0}}]^{\ast}(t))^{p}%
ds\leq2c\alpha=2c\sum_{j}\mu(B_{j})^{q}\frac{1}{\mu(B_{j})^{q}}\int_{B_{j}%
}|f-f_{B_{j}}|^{p}d\mu(x).
\]
By (4) of Lemma \ref{le:covering}, the set $B_{0}\cap B_{j}\setminus F$ is
nonempty and, therefore, we can find a point $x_{j}\in B_{0}\cap
B_{j}\setminus F$. It follows that
\begin{align*}
\int_{0}^{t}([f\chi_{B_{0}}]^{\ast}(s)-[f\chi_{B_{0}}]^{\ast}(t))^{p}ds  &
\leq2c\sum_{j}\mu(B_{j})^{q}[f_{4B_{0},p,q}^{\#}\chi_{B_{0}}(x_{j})]^{p}\\
&  \leq ct^{q}[(f_{4B_{0},p,q}^{\#})^{\ast}(t)]^{p}.
\end{align*}
\end{proof}

\begin{corollary}
(cf. \cite[page 228]{SS}) Suppose that $\mu(X)=\infty,$ and let $c_{1}$ be the
constant of Theorem \ref{th:metric_bi_lhs}. Then
\begin{equation}
t^{-q/p}\left(  \int_{0}^{t}(f^{\ast}(s)-f^{\ast}(t))^{p}ds\right)  ^{1/p}\leq
c_{1}(f_{X,p,q}^{\#})^{\ast}(t), \label{eq:metric_bi_lhs_global}%
\end{equation}
for all $f\in L_{\text{loc}}^{1}(X),$ $t>0.$
\end{corollary}

\begin{proof}
Let $t>0,$ and let $c_{2}$ be as in Theorem \ref{th:metric_bi_lhs}. Fix an
arbitrary $x_{0}\in X,$ since $\mu(X)=\infty,$ we can find a positive integer
$n_{0}$ such that, for $n\geq n_{0},$ and $B_{n}:=B(x_{0},n),$ we have
$t<c_{2}\mu(B_{n})$. Therefore, by (\ref{eq:metric_bi_lhs}),%
\[
t^{-q/p}\left(  \int_{0}^{t}([f\chi_{B_{n}}]^{\ast}(s)-[f\chi_{B_{n}}]^{\ast
}(t))^{p}ds\right)  ^{1/p}\leq c_{1}(f_{4B_{n},p,q}^{\#})^{\ast}(t),
\]
and consequently
\begin{equation}
t^{-q/p}\left(  \int_{0}^{t}([f\chi_{B_{n}}]^{\ast}(s)-f^{\ast}(t))_{+}%
^{p}ds\right)  ^{1/p}\leq c_{1}(f_{X,p,q}^{\#})^{\ast}(t).
\end{equation}
Letting $n\rightarrow\infty,$ and using Fatou's lemma, we see that
\[
t^{-q/p}\left(  \int_{0}^{t}(f^{\ast}(s)-f^{\ast}(t))^{p}ds\right)  ^{1/p}\leq
c_{1}(f_{X,p,q}^{\#})^{\ast}(t).
\]
\end{proof}

Once the inequality (\ref{eq:metric_bi_lhs}) is available then it can be
combined with the Poincar\'{e} inequality, as described in the introduction,
to obtain the symmetrization inequality.

\begin{theorem}
\label{th:metric_bi} Let $B_{0}\subset X$ be a ball, and suppose that $f$ and
$g$ satisfy a $p-q-$Poincar\'{e} inequality on $4\sigma B_{0}$ (with constant
$c_{P}$). Then there exist positive constants $c_{1}=c_{1}(B_{0},c_{P})$ and
$1\geq c_{2}=c_{2}(X),$ such that, for $0<t<c_{2}\mu(B_{0})$,%
\begin{equation}
t^{-\frac{1}{s}}\left(  \frac{1}{t}\int_{0}^{t}([f\chi_{B_{0}}]^{\ast
}(s)-[f\chi_{B_{0}}]^{\ast}(t))^{p}ds\right)  ^{1/p}\leq c_{1}[(g^{q}%
)^{\ast\ast}(t)]^{1/q}. \label{eq:metric_bi}%
\end{equation}
\end{theorem}

\begin{proof}
From the underlying Poincar\'{e} inequality and (\ref{eq:meas_growth}) (with
$\widetilde{B}=4B_{0})$ we get
\[
\left(  \frac{1}{(\mu(B))^{1+p/s}}\int_{B}|f-f_{B}|^{p}d\mu\right)  ^{1/p}\leq
c\left(  \frac{1}{\mu(\sigma B)}\int_{\sigma B}g^{q}d\mu\right)  ^{1/q},
\]
for every ball $B$ with $B\subset4B_{0}$.

Fix an arbitrary point $x\in B_{0}$. Taking a supremum over all balls
containing $x$ on the right hand side, and over all balls $B\subset4B_{0}$
containing $x$ on the left hand side, we arrive at
\[
f_{4B_{0},p,1+p/s}^{\#}(x)\leq c\left(  Mg^{q}(x)\right)  ^{1/q},
\]
where $M$ is the maximal operator of Hardy-Littlewood. After passing to
rearrangements, and using (recall that the underlying measure is doubling)
\[
(Mh)^{\ast}(t)\leq ch^{\ast\ast}(t),
\]
combined with Theorem \ref{th:metric_bi_lhs}, we obtain positive constants
$c_{1}$, $c_{2}$ such that
\[
t^{-\frac{1}{s}}\left(  \frac{1}{t}\int_{0}^{t}([f\chi_{B_{0}}]^{\ast
}(s)-[f\chi_{B_{0}}]^{\ast}(t))^{p}ds\right)  ^{1/p}\leq c_{1}[(g^{q}%
)^{\ast\ast}(t)]^{1/q},0<t<c_{2}\mu(B_{0}).
\]
\end{proof}

\begin{remark}
\label{re:counterexample_bi} It may not be possible to extend the inequality
(\ref{eq:metric_bi}) to all $0<t<\mu(B_{0}).$ This can be seen from the
following counterexample for $p=q=1$.

Let $f_{k}(x)=1$ if $x\in\lbrack-1,1]^{2}\setminus B_{1/k}(0),$ $f_{k}%
(x)=k|x|$ for $x\in B_{1/k}(0)$. Now, $|\nabla f_{k}|^{\ast}(t)=k\chi
_{(0,\frac{\pi}{k^{2}}]}(t)$ and $|\nabla f_{k}|^{\ast\ast}(t)=k\chi
_{(0,\frac{\pi}{k^{2}}]}(t)+\frac{\pi}{k}\frac{1}{t}\chi_{\lbrack\frac{\pi
}{k^{2}},\infty)}(t)$. If (\ref{eq:metric_bi}) were true for $0<t<4$, then
taking the limit as $t\rightarrow4$ of (\ref{eq:metric_bi}) would give us%
\[
\Vert f_{k}\Vert_{L^{1}(Q)}\leq c\frac{\pi}{4}\frac{1}{k}.
\]
But whereas the right-hand side $\rightarrow0$ as $k\rightarrow\infty$ the
left-hand side $\approx4$ . One possible way to overcome this problem is to
consider functions with zero average, by means of replacing $f$ by $f-f_{Q}$
(cf. \cite{JMMM}).
\end{remark}

\begin{remark}
\label{re:global_bi} Suppose that the following global growth condition holds
for every ball $B\subset X,$%
\begin{equation}
\mu(B)\geq cr(B)^{s}. \label{eq:global_growth}%
\end{equation}
Then using the proof of Theorem \ref{th:metric_bi} together with
(\ref{eq:metric_bi_lhs_global}) yields
\[
t^{-\frac{1}{s}}\left(  \frac{1}{t}\int_{0}^{t}(f^{\ast}(s)-f^{\ast}%
(t))^{p}ds\right)  ^{1/p}\leq c_{1}[(g^{q})^{\ast\ast}(t)]^{1/q},t>0.
\]
\end{remark}

\section{Applications}

First we consider the Sobolev embedding theorem for metric spaces.

\begin{theorem}
\label{th:metric_embedding} Let $B_{0}\subset X$ be a ball and let $Y=Y(X)$ be
an r.i. space. Suppose that the operator $P_{\max\{p,q\}}$ is bounded on $Y,$
and let $f$ and $g$ satisfy a $p-q-$Poincar\'{e} inequality on $4\sigma B_{0}%
$. Then there exists a constant $c>0,$ independent of $f$ and $g$, such that
\[
\Vert f\chi_{B_{0}}\Vert_{Y^{p}(\infty,s)}\leq c(\Vert g\Vert_{Y}+\Vert
f\chi_{B_{0}}\Vert_{Y}).
\]
\end{theorem}

\begin{proof}
By Theorem \ref{th:metric_bi} there are constants $c_{1},c_{2}$ such that, for
$0<t<c_{2}\mu(B_{0}),$%
\[
t^{-\frac{1}{s}}\left(  \frac{1}{t}\int_{0}^{t}([f\chi_{B_{0}}]^{\ast
}(s)-[f\chi_{B_{0}}]^{\ast}(t))^{p}ds\right)  ^{1/p}\leq c_{1}[(g^{q}%
)^{\ast\ast}(t)]^{1/q}.
\]
If $t\geq c_{2}\mu(B_{0})$, we have
\[
t^{-\frac{1}{s}}\left(  \frac{1}{t}\int_{0}^{t}([f\chi_{B_{0}}]^{\ast
}(s)-[f\chi_{B_{0}}]^{\ast}(t))^{p}ds\right)  ^{1/p}\leq(c_{2}\mu
(B_{0}))^{-\frac{1}{s}}\left(  \frac{1}{t}\int_{0}^{t}([f\chi_{B_{0}}]^{\ast
}(s))^{p}ds\right)  ^{1/p}.
\]
Therefore,%
\[
t^{-\frac{1}{s}}\left(  \frac{1}{t}\int_{0}^{t}([f\chi_{B_{0}}]^{\ast
}(s)-[f\chi_{B_{0}}]^{\ast}(t))^{p}ds\right)  ^{1/p}\leq c\left[
P_{q}(g)(t)+P_{p}(f\chi_{B_{0}})(t)\right]  ,
\]
for all $t>0$. In view of our assumption on $P_{\max\{p,q\}}$ it follows, upon
applying the $Y$ norm to both sides of the previous inequality, that
\[
\Vert t^{-\frac{1}{s}}\left(  \frac{1}{t}\int_{0}^{t}([f{\chi_{B_{0}}}]^{\ast
}(s)-[f{\chi_{B_{0}}}]^{\ast}(t))^{p}ds\right)  ^{1/p}\Vert_{Y}\leq c\left[
\Vert g\Vert_{Y}+\Vert f\chi_{B_{0}}\Vert_{Y}\right]  ,
\]
as we wished to show.
\end{proof}

\begin{remark}
Arguing as in Remark \ref{re:global_bi} we conclude that, if the global growth
condition (\ref{eq:global_growth}) holds for all balls $B\subset X,$ then, for
all $f$ and $g$ as in Theorem \ref{th:metric_embedding}, we have
\[
\Vert f\Vert_{Y^{p}(\infty,s)}\leq c\Vert g\Vert_{Y}.
\]
\end{remark}

We now apply our symmetrization inequality to derive functional forms of
Faber-Krahn inequalities (see \cite{Ba} for a brief introduction to
inequalities of this type). In the following we denote
\[
\|f\|_{0}:=\mu(\text{supp}(f)).
\]
We will also assume that $\mu$ is nonatomic.

\begin{theorem}
\label{th:metric_faber} Let $B_{0}\subset X$ be a ball and let $f$ be a
function with $\text{supp}(f)\subset B_{0}$. Let $Z(X)$ be an r.i. space and
let $\phi_{Z^{\prime}}$ denote the fundamental function of its associate space
$Z^{\prime}(X).$

(i) Let $f$ be a $p-q-$Sobolev function and let $g$ be a gradient of $f$. If
$\Vert f\Vert_{0}<c_{2}\mu(B_{0})$, where $c_{2}$ is the constant of Theorem
\ref{th:metric_bi}, then
\begin{equation}
\Vert f\Vert_{L^{p}(X)}\leq c\left[  \Vert g^{q}\Vert_{Z(X)}\phi_{Z^{\prime}%
}(\Vert f\Vert_{0})\right]  ^{1/q}\Vert f\Vert_{0}^{\frac{s+p}{sp}-\frac{1}%
{q}}.
\end{equation}

(ii) Let $f$ be a $1-q-$Sobolev function, $q>s,$ and let $g$ be a gradient of
$f$. If $\Vert f\Vert_{0}<c_{2}\mu(B_{0})$, where $c_{2}$ is the constant of
Theorem \ref{th:metric_bi}, then
\begin{equation}
\Vert f\Vert_{L^{\infty}(X)}\leq c\left[  \Vert g^{q}\Vert_{Z(X)}%
\phi_{Z^{\prime}}(\Vert f\Vert_{0})\right]  ^{1/q}\Vert f\Vert_{0}^{1/s-1/q}.
\end{equation}
\end{theorem}

\begin{proof}
\textit{(i)} By Theorem \ref{th:metric_bi} we have, for $0<t<c_{2}\mu(B_{0})$,%
\begin{equation}
t^{-\frac{s+p}{sp}}\left(  \int_{0}^{t}[f^{\ast}(s)-f^{\ast}(t)]^{p}ds\right)
^{1/p}\leq c_{1}[(g^{q})^{\ast\ast}(t)]^{1/q}. \label{eq:proof_fk1}%
\end{equation}
Now, since $\Vert f\Vert_{0}<c_{2}\mu(B_{0}),$ using right-continuity of the
decreasing rearrangement we can substitute $t=\Vert f\Vert_{0}$ in
(\ref{eq:proof_fk1}). Thus,%
\[
\Vert f\Vert_{0}^{-\frac{s+p}{sp}}\Vert f\Vert_{L^{p}}\leq c_{1}\left[
\frac{1}{\Vert f\Vert_{0}}\int_{0}^{\Vert f\Vert_{0}}(g^{q})^{\ast
}(s)ds\right]  ^{1/q}.
\]
Applying H\"{o}lder's inequality we finally obtain%
\[
\Vert f\Vert_{L^{p}}\leq c_{1}\left[  \Vert g^{q}\Vert_{Z}\phi_{Z^{\prime}%
}(\Vert f\Vert_{0})\right]  ^{1/q}\Vert f\Vert_{0}^{\frac{s+p}{sp}-\frac{1}%
{q}}.
\]
\textit{(ii)} We first observe that $-\frac{d}{dt}f^{\ast\ast}(t)=[f^{\ast
\ast}(t)-f^{\ast}(t)]/t$. Thus, by Theorem \ref{th:metric_bi}, we have
\[
-\frac{d}{dt}f^{\ast\ast}(t)\leq c_{1}t^{1/s-1}\left(  \frac{1}{t}\int_{0}%
^{t}(g^{q})^{\ast}(s)ds\right)  ^{1/q}.
\]
Integrating over $(0,\Vert f\Vert_{0})$ yields
\[
\Vert f\Vert_{L^{\infty}(X)}-f^{\ast\ast}(\Vert f\Vert_{0})\leq c_{1}\int
_{0}^{\Vert f\Vert_{0}}t^{1/s-1-1/q}\left(  \int_{0}^{t}(g^{q})^{\ast
}(s)ds\right)  ^{1/q}dt.
\]
Thus, estimating the inner integral using H\"{o}lder's inequality$,$%
\begin{align*}
\Vert f\Vert_{L^{\infty}(X)}  &  \leq c_{1}\left[  \Vert g^{q}\Vert_{Z}%
\phi_{Z^{\prime}}(\Vert f\Vert_{0})\right]  ^{1/q}\int_{0}^{\Vert f\Vert_{0}%
}t^{1/s-1-1/q}dt+f^{\ast\ast}(\Vert f\Vert_{0})\\
&  =c_{1}\left[  \Vert g^{q}\Vert_{Z}\phi_{Z^{\prime}}(\Vert f\Vert
_{0})\right]  ^{1/q}\frac{1}{1/s-1/q}\Vert f\Vert_{0}^{1/s-1/q}+\frac{\Vert
f\Vert_{L^{1}(X)}}{\Vert f\Vert_{0}}\\
&  \leq c\left[  \Vert g^{q}\Vert_{Z}\phi_{Z^{\prime}}(\Vert f\Vert
_{0})\right]  ^{1/q}\Vert f\Vert_{0}^{1/s-1/q},
\end{align*}
where in the last line we used the result obtained in the first half of the theorem.
\end{proof}

Finally for a different connection between Poincar\'{e} inequalities and
symmetrization with other interesting applications we refer to \cite{JMMMR}
(see also \cite{LE} for the relevant family of Poincar\'{e} inequalities). It
would be of interest to extend the results of these papers to the metric setting.

\section{Appendix: Vector fields satisfying H\"{o}rmander's condition}

We present a concrete example of our embedding theorem. But first let us
briefly review some relevant definitions.

Let $X_{1},\dots,X_{m},$ be a collection of $\mathcal{C}^{\infty}$ vector
fields defined in a neighborhood $\Omega$ of the closure of $B(0,1)$, the unit
ball in $\mathbb{R}^{n}$. For a multiindex $\alpha=(i_{1},\dots,i_{k})$ denote
by $X_{\alpha}$ the commutator $[X_{i_{1}},[X_{i_{2}},\dots,[X_{i_{k-1}%
},X_{i_{k}}]],\dots]$ of length $|\alpha|=k$. We shall assume that
$X_{1},\dots,X_{m},$ satisfy H\"{o}rmander's condition: there exists an
integer $d$ such that the family of commutators, up to order $d$,
$\{X_{\alpha}\}_{|\alpha|\leq d}$, spans the tangent space $\mathbb{R}^{n}$ at
each point of $\Omega$. An admissible path $\gamma$ is a Lipschitz curve,
$\gamma:[a,b]\rightarrow\Omega,$ such that there exist functions $c_{i}(t)$,
$a\leq t\leq b$, satisfying $\sum_{i=1}^{m}c_{i}^{2}(t)\leq1,$ and
\[
\gamma^{\prime}(t)=\sum_{i=1}^{m}c_{i}(t)X_{i}(\gamma(t)),
\]
for a.e. $t\in\lbrack a,b]$. A natural metric (the so-called
Carnot-Carath\'{e}odory metric) on $\Omega$ associated to $X_{1},\dots,X_{m},$
is defined by%

\begin{align*}
\varrho(\xi,\nu)=\min\{b\geq0:\text{there is an admissible path }%
\gamma:[0,b]\rightarrow\Omega\\
\text{ such that }\gamma(0)=\xi\text{ and }\gamma(b)=\nu\}.
\end{align*}

Jerison \cite{JE} proved the following general theorem.

\begin{theorem}
\label{th:horm_poincare} For every $1\leq p<\infty,$ there exist a constant
$c>0,$ and a radius $r_{0},$ such that, for every $\xi\in\{x\in\mathbb{R}%
^{n}:|x|<1\},$ and every $r$, $0<r<r_{0}$, for which $B(\xi,2r)=\{\eta
:\varrho(\xi,\eta)<2r\}\subset\Omega$, we have
\[
\left(  \int_{B(\xi,r)}|f(x)-f_{B(\xi,r)}|^{p}dx\right)  ^{1/p}\leq cr\left(
\int_{B(\xi,r)}\left[  \sum_{i=1}^{m}|X_{i}f(x)|\right]  ^{p}dx\right)
^{1/p},
\]
for all $f\in\mathcal{C}^{\infty}(\overline{B(\xi,r)})$, where the integration
is with respect to Lebesgue measure.
\end{theorem}

Moreover, by \cite[\S3 and Theorem 4]{nagel}, for an arbitrary compact set
$K\subset\Omega$, there exist $c>0$ and $r_{0}>0$ such that for all $\xi\in K$
and $\varrho<r_{0}$ we have
\begin{equation}\label{eq:doubling_horm}
|B(\xi,2\varrho)|\leq c|B(\xi,\varrho)|,
\end{equation}
where $|\cdot|$ indicates Lebesgue measure. Therefore, if $\xi\in\Omega$ and $r>0$ such
that $B(\xi,4r)\subset\Omega$ and (\ref{eq:doubling_horm}) hold for every $\varrho<4r$,
then $B(\xi,4r)$ with Lebesgue measure, is a homogeneous space. Thus, combining Theorem
\ref{th:horm_poincare} and the theory of this paper implies the following Sobolev
inequality for vector fields satisfying H\"{o}rmander's condition.

\begin{theorem}
\label{th:horm_emb} There exists a positive constant $c$ such that
\begin{equation}
\Vert f\chi_{B(\xi,r)}\Vert_{Y^{p}(\infty,s)}\leq c\left(  \Vert\sum_{i=1}%
^{m}|X_{i}f|\Vert_{Y}+\Vert f\Vert_{Y}\right)
\end{equation}
for all $f\in\mathcal{C}^{\infty}(\overline{B(\xi,4r)})$.
\end{theorem}

A slightly modified version of Theorem \ref{th:horm_emb} sharpens known
results. For example, starting with $1-1-$Poincar\'{e} inequality for vector
fields satisfying H\"{o}rmander's condition, we can proceed as in the proof of
\cite[Theorem 2]{JMMM} and together with Theorem \ref{th:metric_bi} we obtain
the inequality
\begin{equation}
\Vert f\chi_{B}-f_{B}\Vert_{Y^{1}(\infty,s)}\leq c\Vert\sum_{i=1}^{m}%
|X_{i}f|\Vert_{Y}, \label{eq:horm_emb2}%
\end{equation}
where $B=B(\xi,r)$. According to \cite[Theorem 3.1]{BMR} the space
$Y^{1}(\infty,s)$ is contained in the Hansson-Br\'{e}zis-Wainger space
\[
H_{s}(B):=\left\{  f:\int_{0}^{|B|}\left(  \frac{f^{\ast\ast}(t)}{1+\log
\frac{|B|}{t}}\right)  ^{s}\frac{dt}{t}<\infty\right\}
\]
which, in turn, is known to be contained in the space described in
\cite[Theorem 6.1]{HK}.

For other examples of metric spaces complying with the Poincar\'{e} inequality
condition see \cite{HK}.
\end{document}